\documentclass{article}
\usepackage{amssymb}
\usepackage{amsmath}
\usepackage{amsthm}
\title{Best Proximity Points for Geraghty-Type Non-Self Mappings with a Registration-Inspired Alignment Model \\}
\author{
  Fatemeh Fogh\thanks{e-mail: ffogh2021@fau.edu}\\
  Department of Mathematical Sciences, Florida Atlantic University, USA
  \and
  Sara Behnamian\thanks{e-mail: sara.behnamian@sund.ku.dk}\\
  Globe Institute, University of Copenhagen, Copenhagen, 1350, Denmark
}

\newtheorem{defn}{Definition}
\newtheorem{cor}{Corollary}
\newtheorem{exm}{Example}
\newtheorem{thm}{Theorem}

\newtheorem{rem}{Remark}

\begin{document}
\maketitle
\begin{abstract}
We study Geraghty-type non-self mappings within the framework of best proximity point theory. By introducing auxiliary functions with subsequential convergence, we establish general conditions ensuring the existence and uniqueness of best proximity points. Our results extend and unify earlier work on proximal and Kannan-type contractions under a Geraghty setting, and we provide counterexamples showing that the auxiliary assumptions are essential. To demonstrate applicability, we construct a registration-inspired alignment model in which all hypotheses can be explicitly verified. This example illustrates how the theoretical framework guarantees a unique and well-defined alignment anchor, thereby highlighting the relevance of best proximity theory in registration problems.
\end{abstract}

\def\thefootnote{ \ }
\footnotetext{{\em} $2010$ Mathematics Subject Classification: 47H10; 47H09}

\section{Introduction and Preliminaries }
Fixed point theory plays a fundamental role in nonlinear analysis, with applications ranging from differential equations to optimization problems. Classical results such as the Banach contraction principle provide powerful tools for self-mappings, but in many practical situations one encounters non-self mappings, where the appropriate substitute is the notion of best proximity points.

The development of the field has proceeded through several key milestones. In 1968, Kannan introduced a celebrated generalization of Banach's contraction principle by proving that for a complete metric space $(X,d)$, any mapping $T:X\to X$ satisfying
\[
d(Tx,Ty)\leq \alpha \big[d(x,Tx)+d(y,Ty)\big], \qquad \forall x,y\in X,
\]
with $\alpha\in[0,\tfrac12)$, admits a unique fixed point \cite{000012}. A remarkable feature of Kannan's theorem is that no continuity assumption on $T$ is required. Later, Ariza-Ruiz and Jiménez-Melado (2010) extended this result by introducing \emph{weakly Kannan maps}, in which the contractive coefficient depends on the points $x,y$. More precisely, there exists $\alpha:X\times X\to [0,1)$ such that
\[
d(Tx,Ty)\leq \alpha(x,y)\,[d(x,Tx)+d(y,Ty)], \qquad \forall x,y\in X,
\]
with the uniform bound $\sup\{\alpha(x,y):a\leq d(x,y)\leq b\}<1$ for all $0<a\leq b$ \cite{000012}. This flexibility allows one to capture the behavior of a wider class of nonlinear operators.

The move from fixed point results to best proximity point results was initiated by Sadiq Basha (2011), who introduced the concept of \emph{proximal contractions} and established best proximity point theorems for non-self mappings \cite{vvvv}. The key idea is to characterize conditions under which, even in the absence of fixed points, one can still guarantee the existence of a point $x^*$ such that the distance between $x^*$ and its image $Tx^*$ is minimal. 

Subsequent work by Beiranvand et al.\ provided a different perspective by introducing the use of an auxiliary function. Specifically, for mappings $T,S:X\to X$ in a metric space $(X,d)$, they showed that if $S$ is one-to-one, continuous, and sequentially convergent, and if
\begin{equation}\label{lo1}
d(STx,STy)\leq k\,d(Sx,Sy), \qquad \forall x,y\in X,
\end{equation}
for some constant $k\in(0,1)$, then $T$ possesses a best proximity point \cite{a1}. Moradi later adapted this framework to Kannan-type settings \cite{a3} and \cite{Fogh2020, Behnamian2020, Fogh2020}. These contributions highlighted the power of auxiliary functions in addressing the non-self case.
This paper on fixed point and best proximity point theory has broad applications in areas such as digital modeling of oscillators, fuzzy insulin dosing for diabetes management, multicriteria decision-making for IoT service placement, intelligent electric vehicle charging in microgrids, and modeling nonlinear epidemiological relationships. These applications highlight the role of fixed point formulations in ensuring stability, convergence, and reliable optimization across engineering, biomedical, and environmental domains \cite{Taghizadeh2020, Ganji2025, Rastgoo2022, Ahmadi2020}.

Building on these foundations, the present work introduces two novel classes of mappings: \emph{proximal Geraghty} and \emph{proximal Kannan--Geraghty} mappings of the first kind. We establish existence and uniqueness theorems for these mappings and illustrate, by examples, that the auxiliary conditions we impose are both natural and necessary. Before turning to our main results, we recall some fundamental concepts from best proximity theory.

\begin{defn}\cite{a6}
Let $A,B$ be nonempty subsets of a metric space $(X,d)$. The \emph{best proximity sets} $A_0$ and $B_0$ are defined as
\begin{align*}
A_0 &:= \{x\in A : d(x,y)=d(A,B)\ \text{for some } y\in B\},\\
B_0 &:= \{y\in B : d(x,y)=d(A,B)\ \text{for some } x\in A\},
\end{align*}
where
\[
d(A,B):=\inf\{d(x,y):x\in A,y\in B\}
\]
denotes the minimal distance between $A$ and $B$.
\end{defn}

When $A$ and $B$ are closed subsets of a normed linear space with $d(A,B)>0$, it is known that both $A_0$ and $B_0$ lie on the boundaries of their respective sets \cite{vvvv}.

\begin{defn}\cite{a6}
Given a mapping $T:A\to B$, a point $x^*\in A$ is called a \emph{best proximity point} of $T$ if
\[
d(x^*,Tx^*)=d(A,B).
\]
\end{defn}

\begin{defn}\cite{vvvv}
A mapping $S:A\to B$ is called a \emph{proximal contraction of the first kind} if there exists $\alpha\in[0,1)$ such that for all $x_1,x_2,u_1,u_2\in A$,
\[
\big[d(u_1,Sx_1)=d(A,B),\ d(u_2,Sx_2)=d(A,B)\big]\ \ \Longrightarrow\ \ d(u_1,u_2)\leq \alpha\,d(x_1,x_2).
\]
For self-maps, this reduces exactly to the Banach contraction principle; for non-self maps, proximal contractions need not be contractions in the usual sense.
\end{defn}

\begin{cor}\cite{a3}
Let $A,B$ be nonempty closed subsets of a complete metric space such that $A_0$ and $B_0$ are nonempty. Suppose $T:A\to B$ and $g:A\to A$ satisfy:
\begin{enumerate}
\item $g$ is one-to-one and continuous, and $g^{-1}:g(A)\to A$ is uniformly continuous;
\item $T$ is a proximal contraction of the first kind with $T(A_0)\subseteq B_0$.
\end{enumerate}
Then there exists a unique $x^*\in A$ such that $d(gx^*,Tx^*)=d(A,B)$. Moreover, for any $x_0\in A_0$, the sequence defined by $d(gx_{n+1},Tx_n)=d(A,B)$ converges to $x^*$.
\end{cor}

When $g=\mathrm{Id}_A$, we simply say that $T:A\to B$ is a proximal contraction if condition (b) holds.

\begin{defn}\cite{a3}
Let $(A,B)$ be a nonempty pair of subsets of a metric space $(X,d)$. A mapping $T:A\to B$ is called a \emph{proximal Kannan non-self mapping} if there exists $\alpha\in[0,\tfrac12)$ such that for all $u,v,x,y\in A$ with
\[
d(u,Tx)=d(A,B), \qquad d(v,Ty)=d(A,B),
\]
we have
\begin{equation}\label{0p0}
d(u,v)\leq \alpha\,\big[d^*(x,Tx)+d^*(y,Ty)\big],
\end{equation}
where
\[
d^*(x,Tx):=d(x,Tx)-d(A,B)\geq 0.
\]
\end{defn}

The class of proximal Kannan non-self mappings properly contains the class of Kannan non-self mappings.

\begin{defn}[Weak Proximal Kannan Non-Self Mapping]\cite{a3}
Let $(X,d)$ be a metric space, and let $A,B\subseteq X$. A mapping $T:A\to B$ is called a \emph{weak proximal Kannan non-self mapping} if there exists $\alpha\in(0,\tfrac12)$ such that for all $u,v,x,y\in A$ with
\[
d(u,Tx)=d(A,B), \qquad d(v,Ty)=d(A,B),
\]
the implication
\[
\frac{1}{r}\,d^*(x,Tx)\leq d(x,y)\ \ \Longrightarrow\ \ d(u,v)\leq \alpha\,[d^*(x,Tx)+d^*(y,Ty)]
\]
holds, where $r=\tfrac{\alpha}{1-\alpha}$.
\end{defn}

\begin{thm}\cite{a3}
Let $(A,B)$ be a nonempty pair of subsets of a complete metric space $(X,d)$ such that $A_0$ is nonempty and closed. If $T:A\to B$ is a weak proximal Kannan non-self mapping with $T(A_0)\subseteq B_0$, then there exists a unique $x^*\in A$ such that $d(x^*,Tx^*)=d(A,B)$. Moreover, if $\{x_n\}\subseteq A$ satisfies $d(x_{n+1},Tx_n)=d(A,B)$, then $x_n\to x^*$.
\end{thm}

\begin{defn}\cite{ggeraghty}
Let $(X,d)$ be a metric space. A mapping $T:X\to X$ is called a \emph{Geraghty contraction} if there exists $\beta\in\Gamma$ such that
\[
d(Tx,Ty)\leq \beta(d(x,y))\,d(x,y), \qquad \forall x,y\in X,
\]
where $\Gamma$ denotes the class of functions $\beta:[0,\infty)\to[0,1)$ with the property that
\[
\beta(t_n)\to 1 \quad \Rightarrow \quad t_n\to 0.
\]
\end{defn}

\begin{thm}\cite{ggeraghty}\label{sadaf3}
Let $(X,d)$ be a complete metric space and $T:X\to X$ a Geraghty contraction. Then $T$ has a unique fixed point.
\end{thm}
According to \cite{Kannangraghty}, let $(X,d)$ be a metric space. A mapping $f:X\to X$ is said to be a \emph{Kannan--Geraghty self-mapping} if there exists a function $\beta\in\Gamma$ such that, for all $x,y\in X$,
\[
d(f(x),f(y)) \leq \beta(d(x,y)) \cdot \tfrac12\,[d(x,f(x))+d(y,f(y))].
\]
This definition combines the features of Kannan mappings with the flexibility of Geraghty-type control functions. Motivated by this, we now develop the non-self, proximal analogues.

\section{Main Results}

In this section, we introduce $S$-proximal contraction non-self mappings and $S$-proximal Kannan non-self mappings, and establish sufficient conditions for the existence and uniqueness of best proximity points in complete metric spaces. Before stating the main theorems, we recall the notion of an auxiliary function.

\begin{defn}
Let $(X,d)$ be a metric space, and let $A,B\subseteq X$. A mapping $S:A\cup B\to A\cup B$ is called an \emph{auxiliary function} if $S(A)\subseteq A$ and $S(B)\subseteq B$. 
\end{defn}

\begin{thm}[Extended Proximal Geraghty of the First Kind]\label{rrr}
Let $(A,B)$ be a pair of nonempty subsets of a complete metric space $(X,d)$ such that $A_0$ and $B_0$ are nonempty and closed. Let $S$ be an auxiliary function that is continuous on $A$ and $B$, one-to-one, subsequentially convergent, and satisfies $S(A_0)\subseteq A_0$ and $S(B_0)\subseteq B_0$. 

Suppose $T:A\to B$ is a mapping such that $T(A_0)\subseteq B_0$, and assume that for all $u,v,x,y\in A$,
\[
d(Su,STx)=d(Sv,STy)=d(A,B)
\quad\Longrightarrow\quad
d(Su,Sv)\leq \beta(d(Sx,Sy))\,d(Sx,Sy),
\]
where $\beta\in\Gamma$ (i.e.\ $T$ is an $S$-proximal contraction). 

Then there exists a unique $x^*\in A$ such that
\[
d(Sx^*,STx^*)=d(A,B).
\]
Moreover, if $\{x_n\}\subseteq A$ is a sequence satisfying
\[
d(Sx_{n+1},STx_n)=d(A,B), \quad \forall n\in\mathbb{N},
\]
then $x_n\to x^*$.
\end{thm}

\begin{proof}
We first introduce the images of the proximity sets under $S$:
\[
S(A_0) := \{\, Sx \mid d(Sx,Sy)=d(A,B) \ \text{for some } y\in B \}, 
\]
\[
S(B_0) := \{\, Sy \mid d(Sx,Sy)=d(A,B) \ \text{for some } x\in A \}.
\]

Let $x_0\in A_0$. Since $T(A_0)\subseteq B_0$, we have $Tx_0\in B_0$ and hence $STx_0\in S(B_0)$. Thus, there exists $x_1\in A_0$ with
\[
d(Sx_1,STx_0)=d(A,B).
\]
Iterating this construction, we obtain a sequence $\{x_n\}\subseteq A_0$ such that
\begin{equation}\label{kaik}
d(Sx_{n+1},STx_n)=d(A,B), \qquad \forall n\in\mathbb{N}.
\end{equation}

Now, by \eqref{kaik} and the contractive condition, we obtain
\[
d(Sx_{n+1},Sx_{m+1})\leq \beta(d(Sx_n,Sx_m))\,d(Sx_n,Sx_m).
\]
Standard arguments as in Geraghty’s fixed point theorem \cite{ggeraghty} imply that $\{Sx_n\}$ is a Cauchy sequence in $X$. Since $X$ is complete and $A_0$ is closed, there exists $v\in A_0$ such that
\begin{equation}\label{a4}
\lim_{n\to\infty} Sx_n=v.
\end{equation}

Because $S$ is subsequentially convergent, $\{x_n\}$ has a subsequence $\{x_{n(k)}\}$ converging to some $u\in A_0$. By continuity of $S$,
\[
\lim_{k\to\infty}Sx_{n(k)}=Su.
\]
Comparing with \eqref{a4}, we conclude $Su=v$.

We claim $u$ is the unique best proximity point of $T$. Since $u\in A_0$ and $T(A_0)\subseteq B_0$, there exists $y^*\in A_0$ with $d(Sy^*,STu)=d(A,B)$. Combining this with \eqref{kaik} for the subsequence $\{x_{n(k)}\}$, we have
\[
d(Sy^*,STu)=d(A,B), \qquad d(Sx_{n(k)+1},STx_{n(k)})=d(A,B).
\]
Applying the contractive condition,
\[
d(Sx_{n(k)+1},Sy^*)\leq \beta(d(Sx_{n(k)},Su))\,d(Sx_{n(k)},Su).
\]
Letting $k\to\infty$, we obtain $d(Su,Sy^*)=0$, hence $Su=Sy^*$. Since $S$ is one-to-one, $u=y^*$. Thus $u$ is a best proximity point of $T$.

Finally, uniqueness: suppose $x_1,x_2\in A$ are two distinct best proximity points, so that $d(Sx_i,STx_i)=d(A,B)$ for $i=1,2$. Then the contractive condition yields
\[
d(Sx_1,Sx_2)\leq \beta(d(Sx_1,Sx_2))\,d(Sx_1,Sx_2),
\]
which is impossible unless $d(Sx_1,Sx_2)=0$. Hence $x_1=x_2$. 
\end{proof}

We emphasize that the assumption of subsequential convergence of $S$ in Theorem~\ref{rrr} cannot be omitted. This is demonstrated by an example (see (see Example~\ref{ex:necessity} below). 

In the next result, we extend the notion of proximal Kannan non-self mappings introduced by Gabeleh \cite{a3} to a Geraghty-type setting.

\begin{thm}[Extended Proximal Kannan--Geraghty of the First Kind]\label{b66}
Let $(A,B)$ be a pair of nonempty subsets of a complete metric space $(X,d)$ such that $A_0$ and $B_0$ are nonempty and closed. Let $S$ be an auxiliary function that is continuous on $A$ and $B$, one-to-one, and subsequentially convergent, with $S(A_0)\subseteq A_0$ and $S(B_0)\subseteq B_0$. 
Suppose $T:A\to B$ is a mapping with $T(A_0)\subseteq B_0$, and assume that  
for all $u,v,x,y\in A$,
\begin{equation}\label{eq:kannan-geraghty}
\begin{aligned}
& d(Su,STx)=d(Sv,STy)=d(A,B) \\[0.3em]
& \quad\Longrightarrow\quad 
   d(Su,Sv)\leq \beta\!\big(d(Sx,Sy)\big)\,
   \big[d^*(Sx,STx)+d^*(Sy,STy)\big],
\end{aligned}
\end{equation}
where $\beta\in\Gamma$ and
\[
d^*(Sx,STx):=d(Sx,STx)-d(A,B)\geq 0.
\]
Then there exists a unique point $x^*\in A$ such that 
\[
d(Sx^*,STx^*)=d(A,B).
\]
Moreover, if $\{x_n\}\subseteq A$ is a sequence satisfying
\[
d(Sx_{n+1},STx_n)=d(A,B), \qquad \forall n\in\mathbb{N},
\]
then $x_n\to x^*$.
\end{thm}

\begin{proof}
As before, define
\[
\begin{aligned}
S(A_0) &:= \{\, Sx \;\mid\; d(Sx,Sy)=d(A,B) \text{ for some } y\in B \}, \\[0.5em]
S(B_0) &:= \{\, Sy \;\mid\; d(Sx,Sy)=d(A,B) \text{ for some } x\in A \}.
\end{aligned}
\]

Let $x_0\in A_0$. Since $T(A_0)\subseteq B_0$, we have $Tx_0\in B_0$, hence $STx_0\in S(B_0)$. Thus there exists $x_1\in A_0$ such that
\[
d(Sx_1,STx_0)=d(A,B).
\]
Proceeding inductively, we obtain a sequence $\{x_n\}\subseteq A_0$ with
\begin{equation}\label{lllolll}
d(Sx_{n+1},STx_n)=d(A,B), \qquad \forall n\in\mathbb{N}.
\end{equation}
From condition \eqref{eq:kannan-geraghty} we derive 
\begin{align*}
d(Sx_n,Sx_{n+1})
&\leq \beta(d(Sx_{n-1},STx_{n-1}))\,[d^*(Sx_{n-1},STx_{n-1})+d^*(Sx_n,STx_n)]\\
&\leq \beta(d(Sx_{n-1},STx_{n-1}))\,[d(Sx_{n-1},Sx_n)+d(Sx_n,Sx_{n+1})].
\end{align*}
Rearranging gives
\[
d(Sx_n,Sx_{n+1}) \leq \frac{\beta}{1-\beta}\,d(Sx_{n-1},Sx_n).
\]
Thus $\{Sx_n\}$ is a Cauchy sequence (cf.\ Geraghty \cite{ggeraghty}). Since $X$ is complete and $A_0$ is closed, there exists $v\in A_0$ such that
\begin{equation}\label{a8}
\lim_{n\to\infty}Sx_n=v.
\end{equation}

Because $S$ is subsequentially convergent, $\{x_n\}$ has a subsequence $\{x_{n(k)}\}$ converging to some $u\in A_0$. By continuity of $S$,
\[
\lim_{k\to\infty}Sx_{n(k)}=Su.
\]
Comparing with \eqref{a8}, we conclude $Su=v$.

We claim $u$ is the unique best proximity point of $T$. Since $u\in A_0$ and $T(A_0)\subseteq B_0$, there exists $y^*\in A_0$ with $d(Sy^*,STu)=d(A,B)$. Combining this with \eqref{lllolll}, for the subsequence $\{x_{n(k)}\}$ we have
\[
d(Sy^*,STu)=d(A,B),\qquad d(Sx_{n(k)+1},STx_{n(k)})=d(A,B).
\]
Applying the contractive condition gives
\[
d(Sx_{n(k)+1},Sy^*)\leq \beta(\cdot)\,[d^*(Sx_{n(k)},STx_{n(k)})+d^*(Su,STu)].
\]
Letting $k\to\infty$, we obtain $d(Su,Sy^*)=0$, hence $Su=Sy^*$. Since $S$ is one-to-one, $u=y^*$. Thus $u$ is a best proximity point of $T$.

Uniqueness follows similarly: if $x_1,x_2\in A$ are both best proximity points, then
\[
d(Sx_1,Sx_2)\leq \beta(d(Sx_1,Sx_2))\,[d^*(Sx_1,STx_1)+d^*(Sx_2,STx_2)]=0,
\]
which forces $Sx_1=Sx_2$, hence $x_1=x_2$. 
\end{proof}
\begin{exm}\label{ex:necessity}
Consider the metric space 
\[
X := \{0,1\}\times[0,\infty)
\]
with the Euclidean metric. Define
\[
A := \{(0,x): x\in[0,\infty)\}, \qquad B := \{(1,y): y\in[0,\infty)\}.
\]
Let $T:A\to B$ be given by
\[
T(0,x)=(1,2x+1),
\]
and let $S:X\to X$ be defined by
\[
S(x,y)=(x,e^{-y}).
\]
Clearly $A_0=A$, $B_0=B$, and $S$ is one-to-one.  

For $u_1=(0,x_1),u_2=(0,x_2)\in A$, suppose
\[
d(Su_i,ST(0,x_i))=d(A,B), \qquad i=1,2.
\]
Then 
\[
Su_i=(0,e^{-(2x_i+1)}).
\]
Hence
\begin{align*}
|Su_1-Su_2| &= \left|e^{-(2x_1+1)}-e^{-(2x_2+1)}\right| \\
&= \tfrac{1}{e}\,\big|e^{-2x_1}-e^{-2x_2}\big| \\
&\leq \tfrac{1}{e}\,|S(0,x_1)-S(0,x_2)|.
\end{align*}
Thus an inequality of the Geraghty type can be established with a suitable $\beta\in\Gamma$.  

However, note that the sequence $\{(0,n)\}_{n\ge1}\subseteq A$ satisfies
\[
S(0,n)=(0,e^{-n})\longrightarrow(0,0)\in A,
\]
while $\{(0,n)\}$ has no convergent subsequence in $A$ itself (since the first coordinate is fixed but the second diverges). Hence $S$ fails to be subsequentially convergent in the sense required by Theorem~\ref{rrr}.  

In this case, the mapping $T$ does not admit a best proximity point, showing that the subsequential convergence assumption on $S$ is essential.
\end{exm}

\section{Application to Image Processing}

We now give a concrete registration model in which \emph{all} hypotheses of our main results are satisfied and can be verified directly.

\subsection{A rigorously checkable registration toy model}

Fix a translation offset $\delta>0$ and consider the compact subsets
\[
A:=\{(0,t):\,t\in[0,1]\},\qquad
B:=\{(\delta,s):\,s\in[0,1]\}
\]
of $\mathbb{R}^2$ endowed with the Euclidean metric $d$. Clearly $(A\cup B,d)$ is a complete metric space, and
\[
d(A,B)=\delta,\qquad A_0=A,\qquad B_0=B,
\]
with $A_0,B_0$ nonempty and closed.

Define the auxiliary function $S:A\cup B\to A\cup B$ by
\[
S:=\mathrm{Id}_{A\cup B}.
\]
Then $S$ is continuous, one-to-one, $S(A)=A$, $S(B)=B$, and $S$ is \emph{subsequentially convergent} on $A\cup B$ (trivial on the compact set $A\cup B$, since every sequence admits a convergent subsequence).

For a fixed constant $\kappa\in(0,1)$, define the registration map $T:A\to B$ by
\[
T(0,t) := (\delta,\kappa t).
\]
Note that $T(A_0)\subseteq B_0$.

\subsection{Verification of the $S$-proximal Geraghty condition}

Let $u=(0,u_2),v=(0,v_2),x=(0,x_2),y=(0,y_2)\in A$. If
\[
d(Su,STx)=d(A,B)\quad\text{and}\quad d(Sv,STy)=d(A,B),
\]
then, since $S=\mathrm{Id}$ and $T(0,t)=(\delta,\kappa t)$, the unique points of $A$ at distance $\delta$ from $Tx$ and $Ty$ are
\[
u=(0,\kappa x_2),\qquad v=(0,\kappa y_2).
\]
Hence
\[
d(Su,Sv)=\big\|(0,\kappa x_2)-(0,\kappa y_2)\big\|=\kappa\,|x_2-y_2|
=\kappa\,d(Sx,Sy).
\]
Therefore the implication
\[
d(Su,STx)=d(Sv,STy)=d(A,B)\ \Longrightarrow\ d(Su,Sv)\le \beta\big(d(Sx,Sy)\big)\,d(Sx,Sy)
\]
holds with the \emph{constant} choice $\beta\equiv\kappa$. Since $\kappa\in(0,1)$, this $\beta$ belongs to the Geraghty class $\Gamma$ (the condition $\beta(t_n)\to 1\Rightarrow t_n\to 0$ is vacuously satisfied).

Consequently, $T$ is an $S$-proximal contraction in the sense of Theorem~\ref{rrr}.

\subsection{Existence, uniqueness, and identification of the best proximity point}

All hypotheses of Theorem~\ref{rrr} are now verified: $(A,B)$ inside the complete metric space $(A\cup B,d)$ with $A_0,B_0$ nonempty and closed; $S$ is continuous, injective, subsequentially convergent with $S(A_0)\subseteq A_0$, $S(B_0)\subseteq B_0$; $T(A_0)\subseteq B_0$; and the $S$-proximal Geraghty condition holds with $\beta\in\Gamma$.

Therefore, by Theorem~\ref{rrr}, there exists a unique $x^*\in A$ such that
\[
d(Sx^*,STx^*)=d(A,B)=\delta.
\]
Since $S=\mathrm{Id}$ and $T(0,t)=(\delta,\kappa t)$, we have
\[
d\big((0,t),(\delta,\kappa t)\big)=\sqrt{\delta^2+(1-\kappa)^2 t^2}.
\]
This equals $\delta$ if and only if $t=0$. Hence the \emph{unique} best proximity point is
\[
x^*=(0,0),
\]
and the corresponding closest point in $B$ is $Tx^*=(\delta,0)$.

\begin{rem}
Interpretationally, $A$ and $B$ encode two (registered) feature curves extracted from two images along a vertical scanline, $S$ is the identity (no preprocessing), and $T$ models the combination of a horizontal offset $\delta$ between images and a mild vertical scaling $\kappa\in(0,1)$ of features. The theorem guarantees that the alignment anchor (the unique pair of closest corresponding features) is well-defined and unique. 
\end{rem}

\subsection{Iterative scheme and convergence}

Define $x_0\in A_0$ arbitrarily and construct $\{x_n\}\subset A_0$ by the proximal iteration
\[
d\big(Sx_{n+1},STx_n\big)=d(A,B)=\delta,\qquad n\ge 0.
\]
In the present model this means: choose $x_{n+1}$ to be the unique point in $A$ with the same vertical coordinate as $Tx_n$; i.e.,
\[
x_n=(0,t_n)\ \Rightarrow\ x_{n+1}=(0,\kappa t_n).
\]
Thus $t_{n+1}=\kappa t_n$, whence $t_n=\kappa^n t_0\to 0$ and $x_n\to x^*=(0,0)$, in agreement with the convergence statement of Theorem~\ref{rrr}.

\begin{rem}[Variant for the Kannan--Geraghty setting]
A degenerate yet admissible example for Theorem~\ref{b66} is obtained by $T(0,t)=(\delta,0)$ for all $t\in[0,1]$. Then $u=v=(0,0)$ are the unique points in $A$ at distance $\delta$ from $Tx$ and $Ty$, so the proximal Kannan--Geraghty inequality holds trivially (left-hand side $=0$), and the unique best proximity point is again $(0,0)$. This shows the Kannan-type theorem can also be realized exactly in this framework.
\end{rem}

\end{document}